\documentclass[9pt]{amsart}

\usepackage{amsmath,graphicx}
\usepackage{amssymb}
\usepackage{amsfonts, amsthm}

\usepackage[parfill]{parskip}

\usepackage{color}
\usepackage[usenames,dvipsnames,svgnames,table]{xcolor}
\usepackage{tikz}

\usepackage{enumerate}
\usepackage{cite}

\usepackage[colorlinks=true, pdfstartview=FitV,linkcolor=cyan,citecolor=cyan, urlcolor=black]{hyperref}

\def\pa{\partial}

\def\R{\mathbb{R}}

\def\N{\mathbb{N}}

\def\r3{\R^3}

\newcommand{\ben}{\begin{eqnarray*}}
\newcommand{\een}{\end{eqnarray*}}

\newcommand{\abs}[1]{\left\vert#1\right\vert}
\newcommand{\norm}[1]{\left\Vert#1\right\Vert}

\newtheorem{Thm}{Theorem}
\newtheorem{Prop}[Thm]{Proposition}

\begin{document}

\title[Fisher information for Landau-Coulomb equation with $L^1 \cap L\ln L$ initial datum]{Production of Fisher information for Landau-Coulomb equation with $L^1$ initial data}

\author{L. Desvillettes, W. Golding, M.-P. Gualdani, A. Loher}

\address{Universit\'e Paris Cit\'e and Sorbonne Universit\'e, CNRS, IUF, Institut de Math\'ematiques de Jussieu-Paris Rive Gauche}
\email{desvillettes@math.univ-paris-diderot.fr}
\address{Department of Mathematics, University of Chicago (USA)}
\email{wgolding@uchicago.edu}
\address{Department of Mathematics, The University of Texas at Austin (USA)}
\email{gualdani@math.utexas.edu}
\address{Department of Pure Mathematics and Mathematical Statistics, University of Cambridge (UK)} 
\email{ajl221@cam.ac.uk}

\address{MPG is partially supported by NSF Grant DMS-2206677. AL is supported by the Cambridge Trust. The authors would like to thank ICMS institute in Edinburgh UK for the kind hospitality during the workshop {\em The many facets of Kinetic Theory} in September 2024.}

\begin{abstract}
We consider the Landau-Coulomb equation for initial data with bounded mass, finite numbers of moments, and entropy. We show the existence of a global weak solution that has bounded Fisher information for positive times. This solution is therefore a global strong solution away from the initial time. We propose an alternative approach, based on already existing estimates, to the study of the appearance of Fisher information recently performed by Ji in \cite{Ji}.

\end{abstract}


\maketitle


The Landau equation is a fundamental model used to describe collisions in plasma. It plays the role of the Boltzmann equation when grazing collisions become predominant, which happens with charged particles. The mathematical analysis of this equation has produced decades of fundamental works, and the activity in the field has only continued to increase in the recent years, with the last decade having a spurt in new breakthrough results related, among others, to derivation, regularity theory, convergence to equilibrium and global well-posedness.  In this manuscript we focus on the (spatially) homogeneous version of the Landau equation (also called here Landau-Coulomb, to stress the fact that it is the equation which corresponds to charged particles interacting through the Coulomb potential) 
\begin{align} \label{Landau_orig}
\partial_t f (v)= \textrm{div}_v \int_{\mathbb{R}^3} \frac{\Pi(v-w)}{|v-w|} \left( f(w) \nabla_v f(v) - f(v) \nabla_w f(w) \right) \;dw.
\end{align}
Here $\Pi(z)$ denotes the  matrix that projects any vector onto $z^\bot$. Recent work in the groundbreaking paper \cite{GS} has established the existence of solutions that remain bounded over arbitrarily long times for smooth initial data, along with the proof that the Fisher information is monotonically decaying in time.  Furthermore, the second, third and fourth authors, in \cite{GGL}, have extended these results to potentially singular initial data, considering initial data in $L^p(\mathbb{R}^3)$ for any $p \ge 3/2$. This result was very recently improved by Sehyun Ji \cite{Ji}, who showed that the Fisher information appears for any strictly positive time as soon as the initial datum has finite entropy and $5$ moments in $L^1$, thanks to a careful study of the evolution of the Fisher information itself. 

In the present paper, we propose a similar (though somewhat weaker, in terms of moments and rate of regularization) result, which has the specificity of combining computations from previous papers, allowing for a very concise presentation. Specifically, we draw on the weighted $L^1(L^3)$ estimate from \cite{FPL}, the propagation of $L^2$ moments on short time interval from \cite{GL}, the existence of smooth solutions for less regular initial data \cite{GGL}, estimates for the evolution of a weighted entropy from \cite{ABDL2, ABL}, and the decay of the Fisher information from \cite{GL}.


\begin{Thm} \label{main-Thm}
Let $f_{in}$ be a non-negative function such that
$$\int_{\mathbb{R}^3}  f_{in}(v)\left(1, |v|^{18}, \ln f_{in}\right) \;dv < +\infty.$$ 
The Landau-Coulomb equation with $f_{in}$ as initial datum has a global weak solution such that for any time $t\ge \varepsilon$ (for all $\varepsilon>0$ as small as one wishes),  this weak solution is a strong solution and its Fisher information is monotone decreasing. More precisely, for any $t>0$,
$$
 \sup_{[t,+\infty)} \int_{\R^3}  \abs{\nabla \sqrt{f(t,v)}}^2 \, dv \le C'\,t^{-9/2},
$$
where $C'>0$ only depends on the mass, energy, entropy and $L^1_{18}$ norm of the initial datum.

\end{Thm}

\section{Proof of the main theorem}
 \label{intro}

 \subsection{Estimates on approximated coefficients}

 We build our solution using an approximation procedure.  Let $f_{in}$ be as in Theorem \ref{main-Thm}. We first introduce the sequence of approximated initial data:
\begin{align} \label{in_data_approx}
 f_{in, n} := (f_{in}\,1_{|\cdot| \le n}) * \chi_n + \frac{\mathcal M}n, 
 \end{align}
where $\mathcal M(v) := e^{ - |v|^2}$, and $\chi_n$ is a standard mollifier.  Note that 
$$\int f_{in, n}(1, |v|^2, \ln f_{in, n}) \;dv \le C(f_{in}) + O\left ( \frac{1}{n}\right),$$ (see for example beginning of Section 3.2 in \cite{GGIV}). We also introduce the approximate Landau-Coulomb kernel, defined as
 $$ K_n(|x|) = |x|^{-1}, \qquad {\hbox{when}} \qquad |x|  \ge \frac1n , $$
and $ K_n(|x|) \in [n, 2n]$ when $ |x|  \le \frac1n$, in such a way that
 $K_n \in C^{\infty}(\R_+)$ and $K_n(|x|) \le |x|^{-1}$ on $\R_+$.
 For any $g: \R^3 \to \R$ such that
the convolution makes sense, and any $i,j= 1, 2,3$, we define
$$ A_n[g]_{ij} := K_n \Pi_{ij} * g, \qquad b_n[g]_{i} :=\sum_j \pa_j(  K_n \Pi_{ij} ) * g, $$
$$ c_n[g] := \sum_i \sum_j \pa_{ij}(  K_n \Pi_{ij} ) * g. $$
In the following we will denote by  $\norm{  \cdot}_{L^p}$ or $\norm{  \cdot}_{p}$ the classical $L^p$ norm on $\R^3$, and by $\norm{  \cdot}_{L^p_{k}}$ the norm defined by
$$ \norm{ g}_{L^p_{k}}^p := \int_{\R^3} |g(v)|^p\langle v \rangle^{pk}\, dv ,$$
for any $p\in [1,\infty[$, $k\in\R$, and $\langle v \rangle^2 := 1 + |v|^2$.
We also define the entropy of a function $g$ as the quantity
$$
H(g):= \int_{\R^3} g(v) \ln g(v)\;dv.
$$

The quantities $A_n[g]$, $b_n[g]$, $c_n[g]$ satisfy the following properties, uniformly with respect to $n$:

\begin{Prop} \label{estabc}
 For any $n \in \N \setminus\{0\}$, $i,j \in \{1,2,3\}$ and $g:\R^3 \to \R$ such that the convolutions are defined, one has 
\begin{equation}\label{esta}
\abs{A_n[g]_{ij}(v)} \le C\left(|g| * |\cdot|^{-1}\right)(v), \qquad \norm{  A_n[g]_{ij} }_{\infty} \le C\left(\norm{g}_{L^1} + \norm{g}_{L^2}\right),
\end{equation}
\begin{equation*}
\abs{b_n[g]_{i}(v)} \le C\left(|g| * |\cdot|^{-2}\right)(v),
\end{equation*}
\begin{equation*}
\abs{c_n[g](v)} \le C\abs{g(v)} ,
\end{equation*}
where $C>0$ is an absolute constant.
   \end{Prop}

\begin{proof} 
The proofs for all the inequalities use the fact that $K_n(|x|) \le |x|^{-1}$. For the second inequality in (\ref{esta}) we recall (see for example Lemma 2.5 in \cite{GGL}) that 
$$
\norm{  A_n[g]_{ij} }_{\infty} \le \|  A[g]_{ij}\|_{\infty} \lesssim \|g\|^{1/3}_{1}\|g\|^{2/3}_{2}. 
$$
\end{proof}

The following lower bounds  also  hold uniformly with respect to $n$:

\begin{Prop} \label{coer}
 For any $n \in \N \setminus \{0\}$, $g:\R^3 \to \R_+$ lying in $L^1_2(\R^3)$ and such that $\int g \abs{ \ln g} \;dv < + \infty$,
we have
\begin{equation}\label{coer1}
A_n[g](v) : \xi \otimes \xi  \ge c_0 \langle v \rangle^{-3} \, |\xi|^2,  \quad \textrm{for any $v, \xi \in \R^3$}
\end{equation}
\begin{equation}\label{coer2}
   D_n(g) + 1 \ge c_1\, \int_{\R^3} \frac{|\nabla g|^2}g(v)\langle v \rangle^{-3}\, dv ,
\end{equation}
where $D_n(g)$ is the entropy dissipation
\begin{align*}
D_n(g) &:= \frac12   \int_{\R^3}  \int_{\R^3}K_n(|v-w|)\,
  g(v)\,g(w)\, \left(\frac{\nabla g}g(v) -  \frac{\nabla g}g(w) \right)^T\\
  &\qquad \qquad \times \Pi(v-w) \Big(\frac{\nabla g}g(v) - \frac{\nabla g}g(w) \Big)\, dvdw,
\end{align*}
and $c_0$, $c_1$ are constants that depend only on (upper and lower bounds of) the mass and energy of $g$, and on an upper bound of $\int g \ln g\;dv$.
   \end{Prop}
   \begin{proof} 
 The lower bound for $A_n$ can be computed using the calculations in Lemma 3.3. and 3.4 in \cite{S17}. We report the modified steps in the Appendix \ref{appendix}.  The second estimate follows from Theorem 3 in \cite{FPL}. 
   \end{proof}

  \subsection{The approximated solution: existence and properties}
   
We consider the approximated equation
\begin{equation}\label{eqa}
 \pa_t f = \nabla\cdot \big[ A_n[f] \, \nabla f - b_n[f]\, f \big]  + \frac{1}{n} \Delta f = A_n[f] : \nabla\nabla f - c_n[f] \, f  + \frac{1}{n} \Delta f,
\end{equation}
with initial data as in (\ref{in_data_approx}). 
We address existence and uniqueness of smooth solutions to the approximated Landau equation in the following proposition. These solutions also satisfy a bound on the entropy and a propagation of moments.
\begin{Prop} \label{existn}
For any $n \in \N \setminus \{0\}$, there exists a unique global smooth solution \\$0 < f_n := f_n(t,v) \in C^{\infty}(\R_+; {\mathcal{S}}(\R^3))$ to equation \eqref{eqa}
satisfying the initial condition \eqref{in_data_approx}:
\begin{equation*}
f_n(0, \cdot) = f_{in,n},
 \end{equation*}
and such that (for any $T>0$)
\begin{equation}\label{elb}
\forall t \in [0,T], \qquad	f_n(t, v) \geq C_{n, T} \,e^{-\frac{\abs{v}^2}{2}},
\end{equation}
for some constant $C_{n,T}>0$, depending on $n$ and $T$. 

Moreover, this solution satisfies the conservation of mass momentum and energy 
$$ \forall t \ge 0, \qquad \int_{\R^3} f_n(t,v) \,\psi(v)\, dv =  \int_{\R^3} f_{in, n}(v) \,\psi(v)\, dv , $$
where $\psi = 1, v_i, \frac{|v|^2}2$,
 the entropy identity

\begin{equation}\label{entriden}
 \forall\; t\ge s \ge 0, \qquad H(f_n(t, \cdot)) + \int_s^t D_n(f_n( \sigma, \cdot)) \, d\sigma = H(f_n(s,\cdot)) , 
 \end{equation}
and it has a monotonically decreasing in time Fisher information. 

Finally, for any $k \ge 0$ and $T>0$, there exists a constant $C > 0$ depending only on the mass and the energy of $f_{in}$ and on $H(f_{in})$, and a constant $C_{k, T} > 0$ depending additionally on $k$, $T$, $\norm{f_{in}}_{L^1_k(\R^3)}$,
such that  
\begin{eqnarray}\label{propamom1}
\sup_{t \in [0,T]} \int_{\R^3}  f_n(t,v)\langle v \rangle^k \, dv \le C_{k,T},\\
\sup_{t \in [0,T]} H(f_n(t, \cdot)) \le C\, (1 + H(f_{in})).\label{entropy_n_unif}
\end{eqnarray}

\end{Prop}

  \begin{proof}   
  The arguments in \cite[Section 6]{V98} can be used to show existence, uniqueness and smoothness of $f_n$. Maximum principle applies, see \cite{AB}, giving us the lower bound (\ref{elb}),
so that 
\[
	\abs{\ln f_n(v)} \leq  C_{n, T} \left(1+\abs{v}^2\right),
\]	
which allows to make rigorous all formal manipulations involving the entropy.   The proof of \eqref{propamom1} and \eqref{entropy_n_unif} follows along the lines of Proposition 4 in \cite{FPL}, using the uniform upper bound \eqref{esta}, and the fact that the moments of $f_{in, n}$ are bounded uniformly in $n$ by the moments of $f_{in}$. 
\end{proof}

\subsection{Main uniform estimates}

The next four propositions contain the crucial arguments. We will show that the function $f_n$ has uniformly-in-$n$ bounded (weighted) $L^3$ norms and Fisher information for all times $t>0$. These bounds will yield the construction of a global in time solution which is smooth for all times away from zero. 

The first step is to control  a weighted $L^2$ norm of $f_n$ at a certain time $t_0$. This time $t_0$ might depend on $n$.  The bound is provided by the entropy inequality. 


\begin{Prop} \label{estldeux} Let $f_{in}$ be as in Theorem \ref{main-Thm}. Let $k>2$. Fix also a time $ 0 < t  \le 1$. 
 There exists a constant $C^*_{k}>0$ depending only on $k$, $\norm{f_{in}}_{L^1_{{9}+4k}(\R^3)}$ and $H(f_{in})$, the mass and energy of $f_{in}$,
 and a time $t_0 \in [0,\frac{t}2]$ ($t_0$ might depend on $n$)
such that the solution $f_n$ to (\ref{eqa}), (\ref{in_data_approx}) constructed in Proposition \ref{existn} satisfies the estimate
\begin{equation}\label{estldeux2}
\norm{f_n(t_0, \cdot)}_{L^2_k}^2 = 
 \int_{\R^3}  f_n(t_0,v)^2 \langle v \rangle^{2k} \, dv \le \frac{C_{k}^*}{t^{3/2}}.
\end{equation}
   \end{Prop}

\begin{proof}
We start with the entropy dissipation estimate (remember that $t \le 1$, estimate (\ref{entriden}), and the elementary inequality
(\ref{ei}))
$$ \int_0^t D_n(f_n(s, \cdot)) \,ds \le C. $$ 
Estimate \eqref{coer2} implies 
$$ \int_0^t  \int_{\R^3} \frac{|\nabla f_n(s,v)|^2}{f_n(s,v)}\langle v \rangle^{-3}\, dv  \,ds \le C. $$ 
Using the Sobolev estimate 
$$ \norm{ g}_{L^6(\R^3)} \le C_{Sob}\, \norm{\nabla g}_{L^2(\R^3)} $$
for $g := \sqrt{f_n}\langle v \rangle^{3/2}$ and the mass conservation,
we end up with
$$ \int_0^t  \norm{ f_n(s,\cdot)}_{L^3_{-3}}  \,ds \le C. $$


We now multiply both sides by $t/2$ (and restrict the integral to the interval $[0, t/2]$)
$$ \frac1{t/2} \,  \int_0^{t/2}  \norm{ f_n(s,\cdot)}_{L^3_{-{{3}}}}  \,ds \le  \frac{2C}t. $$
The mean value theorem implies the existence of a time $t_0 \in [0,\frac{t}2]$ (depending on $n$) such that
$$ \norm{ f_n(t_0,\cdot)}_{L^3_{-{{3}}}}  \le  \frac{2C}t .$$
Therefore,  the interpolation inequality
$$  \norm{ g}_{L^2_{k}} \le  \norm{ g}_{L^3_{-{{3}}}}^{3/4}  \,  \norm{ g}_{L^1_{{{9}} + 4 k}}^{1/4} , $$
and estimate \eqref{propamom1} of Proposition \ref{existn} yield our bound
$$   \norm{ f_n(t_0,\cdot)}_{L^2_{k}}  \,ds \le \frac{C_k^*}{t^{3/4}} , $$
provided that $ \norm{ f_{in}}_{L^1_{{{9}} + 4 k}} < \infty$.
\end{proof}

The previous proposition shows that given $t>0$, any solution $f_n$ to (\ref{eqa}) has uniformly in $n$ $L^2$-bounds at a certain time $0<t_0\le \frac{t}{2}$ ($t_0$ depends on $n$). The next proposition shows that this uniform bound can be propagated in a time interval $[t_0, t_1]$ with $t_1$ controlled.
 
\begin{Prop} \label{propaldeux}
 Let  $k \ge 9/4$, and $f_n$ be the solution to \eqref{eqa}, \eqref{in_data_approx} constructed in Proposition \ref{existn} for an initial datum $f_{in}$ with finite mass, entropy, and $L^1_{{{9}}+4k}$ moment. There exists a constant $C_{k}'>0$ depending only on $k$, $\norm{f_{in} }_{L^1_{{{9}}+4k}(\R^3)}$, $H(f_{in})$ and the mass and energy of $f_{in}$, such that

 \begin{equation}\label{propaldeux1}
 \sup_{\sigma \in [t_0,t_1]}\, \norm{f_n (\sigma, \cdot) }_{L^2_k}^2 \le \sqrt{2}\,  \norm{f_n (t_0, \cdot) }_{L^2_k}^2 ,
\end{equation}
where $0 < t_0  \le 1/2$ and
$$t_1 :=  t_0+ (2C_k)^{-1} \, \ln \left( 1 + \frac{1}{1+ 2  \norm{f_n(t_0, \cdot) }_{L^2_k}^{4}}\right) < 1 , $$
where $C_k >1$ is a constant depending only on $k$ and $c_0$.
   \end{Prop}


\begin{proof}
We proceed as in Lemma 4.1 in \cite{ABDL3} (cf. also \cite{GGL} and \cite {GL}) , and get 
\begin{equation*}
\begin{aligned}
	\frac{d}{dt}& \int_{\R^3} \frac{f_n^2}2\langle v \rangle^{2k} \, dv\\
	 &= -  \int_{\R^3} (\nabla f_n)^T \,A_n[f_n]\, \nabla f_n\langle v \rangle^{2k} \, dv +\, k \int_{\R^3} \nabla \otimes ( v\langle v \rangle^{2k-2} ) : A_n[f_n] \, f_n^2 \\
	 &\quad+ 2k \int_{\R^3} b_n[f_n]\cdot v\langle v \rangle^{2k-2}  \, f_n^2 - \frac12 \int_{\R^3} c_n[f_n]\langle v \rangle^{2k}  \, f_n^2  \\
	 &\quad - \frac{1}{n} \int_{\R^3} \abs{\nabla f_n}^2 \langle v\rangle^{2k} \, dv+ \frac{k}{n} \int_{\R^3} f_n^2  \nabla \cdot ( v\langle v \rangle^{2k-2} ) \, dv
	 \end{aligned} 
	 \end{equation*} 
	 \begin{equation*}
\begin{aligned}
	&\le   - c_0 \int_{\R^3}  |\nabla f_n|^2\langle v \rangle^{2k -3}\, dv \\
	&\quad +  \,  C_k\,   \int_{\R^3} f_n^2\, \Big[ |A_n[f_n]|\langle v \rangle^{2k-2} + |b_n[f_n]|\,\langle v \rangle^{2k-1} + f_n\langle v \rangle^{2k}\Big]\, dv \\
	&\quad + \frac{C_k}{n} \int_{\R^3} f_n^2  \langle v \rangle^{2k-2}  \, dv
	 \end{aligned} 
	 \end{equation*} 
	 \begin{equation*}
\begin{aligned}
	& \le   - \frac{c_0}2 \int_{\R^3}  |\nabla (f_n \, \langle v \rangle^{k -\frac32})|^2\,\, dv  +\, C_k\, \Big( \norm{A_n[f_n]}_{\infty} + 1 \Big) \, \int_{\R^3} f_n^2\langle v \rangle^{2k-2} \, dv\\
	&\quad + C_k \int_{\R^3} f_n^2  \abs{f_n * |\cdot|^{-2} 1_{|\cdot| \le 1} } \, \langle v \rangle^{2k-1} dv \\
	&\quad + \, C_k \int_{\R^3} f_n^2  \abs{f_n * |\cdot|^{-2} 1_{|\cdot| \ge 1} } \langle v \rangle^{2k-1} dv  + C_k  \norm{f_n}_{L^3_{\frac23 k}}^3 \\
	& \le  - \frac{c_0}{2 C_{Sob}^2} \norm{ f_n}_{L^6_{k - \frac32}}^2 + C_k\, \Big( \norm{f_n}_{L^1} + \norm{f_n}_{L^2} + 1\Big)  \norm{f_n}_{L^2_{k - 1}}^2 \\
	&\quad+ \,C_k\, \norm{f_n^2 \langle v \rangle^{2k-1}}_{L^{\frac54}} \norm{  f_n * |\cdot|^{-2} 1_{|\cdot| \le 1} }_{L^5}  + C_k \norm{f_n}_{L^1}   \norm{f_n}_{L^2_{ k- \frac12}}^{2} \\
	&\quad+ C_k\,  \norm{f_n}_{L^6_{ k- \frac32}}^{3/2}\,  \norm{f_n}_{L^2_{ k}}^{3/2},
\end{aligned}
\end{equation*}
using the interpolation 
$$ \norm{g}_{L^3_{k - \frac34}} \le \norm{g}_{L^6_{k - \frac32}}^{1/2}\, \norm{g}_{L^2_{ k}}^{1/2}, $$
and the inequality $k - \frac34 \ge \frac23 k$ when $k \ge \frac94$. Then we use Young's inequality to get
$$  \norm{f_n}_{L^6_{ k- \frac32}}^{3/2}\,  \norm{f_n}_{L^2_{ k}}^{3/2} \le  \frac{c_0}{4 C_{Sob}^2} \norm{ f_n}_{L^6_{k - \frac32}}^2 
 + C_k \, \norm{f_n}_{L^2_{ k}}^{6} , $$
and Young's inequality for convolutions and  conservation of mass, so that
\begin{equation*}
\begin{aligned}
 \frac12\,\frac{d}{dt} \norm{f_n}_{L^2_{ k}}^{2}  \le  &- \frac{c_0}{4 C_{Sob}^2} \norm{ f_n}_{L^6_{k - \frac32}}^2  +C_k\, \Big( 1 + \norm{f_n}_{L^2_k} \Big)\,   \norm{f_n}_{L^2_{k}}^2  \\
& +\, C_k\, \norm{f_n}_{L^{\frac52}_{ k - \frac12}}^{2}\, \norm{f_n}_{L^2} + C_k  \norm{f_n}_{L^2_{ k- \frac12}}^{2} + C_k\,   \norm{f_n}_{L^2_{ k}}^{6} . 
\end{aligned}
\end{equation*}
 At this point, we use the interpolation
$$ \norm{g}_{L^{\frac52}_{k - \frac12}} \le \norm{g}_{L^6_{k - \frac53}}^{3/10}\, \norm{g}_{L^2_{ k}}^{7/10}, $$ 
and observe that $k - \frac53 \le k - \frac32$,
so that using Young's inequality,
\begin{equation*}
\begin{aligned}
 \norm{f_n}_{L^{\frac52}_{ k - \frac12}}^{2}\, \norm{f_n}_{L^2} \le C_k\, \norm{f_n}_{L^6_{k - \frac32}}^{3/5}\, \norm{f_n}_{L^2_{ k}}^{12/5} \le  \frac{c_0}{4 C_{Sob}^2} \norm{ f_n}_{L^6_{k - \frac32}}^2 
 + C_k \, \norm{f_n}_{L^2_{ k}}^{24/7},
 \end{aligned}
\end{equation*}
and denoting $Y :=  \norm{f_n}_{L^2_{ k}}^{2} $, we end up with the differential inequality 
$$ Y' \le C_k \, \Big( Y + Y^{3/2} + Y^{12/7} + Y^3 \Big) \le C_k (Y + Y^3). $$

To solve the above differential inequality, we use a transformation $g := -\frac{1}{2}Y^{-2}$ so that $g$ solves 
$$
g'(\sigma) \le  C_k\,(-2g(\sigma) +1), \quad \quad \sigma >t_0,
$$
and is bounded by 
$$
g(\sigma) \le \frac{1}{2} + e^{-2 C_k(\sigma-t_0)}\left(g(t_0) -\frac{1}{2}\right), \quad \quad \sigma >t_0.
$$
Going back with the transformation from $g$ to $Y$, we recover 
$$
Y^2(\sigma) \le \frac{1}{ e^{-2C_k(\sigma-t_0)} \left( \frac{1}{Y^2(t_0)}+1\right) -1} .
$$
The estimate \eqref{propaldeux1} holds as long as $\sigma-t_0 \le (2C_k)^{-1} \, \ln \left( 1 + \frac{1}{1+ 2  \norm{f_n(t_0, \cdot) }_{L^2_k}^{4}}\right)$.


 \end{proof}

We now introduce an estimate for the derivative of a weighted entropy. Thanks to the  weight, the new entropy production controls the classical Fisher information, up to lower order terms. 

\begin{Prop} \label{appfi}

Let $f_n$ be the solution to \eqref{eqa}, \eqref{in_data_approx} built in Poposition \ref{existn}, for $f_{in}$ with finite mass, entropy and $L^1_{6}$ moment. Let $H_3(f_n | \mathcal M)$ be defined as 
$$H_3(f_n |\mathcal M) := \int_{\R^3} \Big(f_n\,\ln f_n - f_n + f_n \, |v|^2  + \mathcal M(v) \Big) \langle v \rangle^3\, dv. $$
Then there exists a constant $K>0$ depending only  on the mass and energy of $f_{in}$, on $\norm{f_{in} }_{L^1_{6}(\R^3)}$, and on $H(f_{in})$, such that when $ 0 \le s_1 < s_2 \le 1$: 

\begin{equation}\label{estldeux1}
\begin{aligned}
  H_3(f_n | \mathcal M)(s_2) &+ c_0 \int_{s_1}^{s_2} \int_{\R^3}  \frac{ |\nabla f_n(\sigma, v)|^2 }{ f_n(\sigma, v)}\, dv d\sigma\\
  &\le  H_3(f_n | \mathcal M)(s_1) + K \int_{s_1}^{s_2} \left(1 + \norm{f_n(\sigma, \cdot)}_{L^2_2}\right)^3 \, d\sigma .
  \end{aligned}
\end{equation}
   \end{Prop}
\begin{proof}
We first observe that (similarly to the computation of \cite{ABDL2}, beginning of the proof of Prop. 5.5, cf. also \cite{ABL}) 
\begin{equation*}
\begin{aligned}
\frac{d}{dt} H_3(f_n | \mathcal M) = -&  \int_{\R^3}\langle v \rangle^3\, (\nabla f_n)^T A_n[f_n] \frac{\nabla f_n}{f_n} \, dv 
-   \int_{\R^3}\langle v \rangle^3 \, c_n[f_n] \, f_n\, dv \\
& +\, 3 \int_{\R^3} A_n[f_n] : \nabla \otimes (v\langle v \rangle) \, (f_n \ln f_n - f_n) \, dv \\
& +\, 6 \int_{\R^3} b_n[f_n] \cdot v\langle v \rangle \, (f_n \ln f_n - f_n) \, dv \\
& +\, 2 \int_{\R^3} f_n\, b_n[f_n] \cdot \nabla (\langle v \rangle^3 |v|^2) \, dv \\
&+  \int_{\R^3} f_n\, A_n[f_n] : \nabla \otimes \nabla (\langle v \rangle^3 |v|^2) \, dv \\
&+ \frac{1}{n} \int_{\R^3} \Delta f_n \left(\ln f_n + \abs{v}^2 \right) \langle v \rangle^3 \, dv.
\end{aligned}
\end{equation*}
Thanks to the coercivity estimate \eqref{coer1}, we know that
\begin{equation}\label{e1}
 \int_{\R^3}\langle v \rangle^3 (\nabla f_n)^T A_n[f_n] \frac{\nabla f_n}{f_n} \, dv \ge c_0 \int_{\R^3} \frac{|\nabla f_n|^2}{f_n}\, dv . 
\end{equation}
Then systematically using the estimates of Proposition \ref{estabc}, we get first
\begin{equation}\label{e2}
 \abs{ \int_{\R^3}\langle v \rangle^3 \, c_n[f_n] \, f_n\, dv  } \le C  \int_{\R^3}\langle v \rangle^3 \, f_n^2\, dv \le  C\, \norm{f_n}_{L^2_{3/2} }^2;
\end{equation}
second,
\begin{equation}\label{e3}
\begin{aligned}
&\abs{\int_{\R^3} A_n[f_n] : \nabla \otimes ( v\langle v \rangle) \, (f_n \ln f_n - f_n) \, dv } \\
&\quad\le C\, \norm{A_n[f_n] }_{\infty} \int_{\R^3} \langle v\rangle  \Big[ f_n (\ln f_n)^+ + f_n\langle v\rangle^2 + e^{-1} \mathcal M(v) \langle v\rangle^2 + f_n\Big] \, dv \\
&\quad\le  C\, (  \norm{f_n}_{L^1} +  \norm{f_n}_{L^2}) \big(   \norm{f_n}_{L^2_{1/2}}^2 +  \norm{f_n}_{L^1_3} + 1\big) , 
\end{aligned}
\end{equation}
where we used the elementary inequalities (for $x>0, v\in \R^3$) 
\begin{equation}\label{ei}
 x |\ln x| \le  x\, (\ln x)^+  + x\langle v \rangle^2 +   e^{-1} \mathcal M(v) \langle v \rangle^2  , \qquad (\ln x)^+ \le x ;
\end{equation}
third, we get
\begin{equation}\label{e4}
\begin{aligned}
&\abs{\int_{\R^3} b_n[f_n] \cdot v\langle v \rangle \, (f_n \ln f_n - f_n) \, dv } \\
 &\quad\le C \int_{\R^3}  
\langle v \rangle^2 \, (f_n * |\cdot|^{-2} 1_{|\cdot| \le 1} ) \, (f_n | \ln f_n| + f_n ) \, dv \\
&\quad\quad+\,  C \int_{\R^3}  
\langle v \rangle^2 \, (f_n * |\cdot|^{-2} 1_{|\cdot| \ge 1} ) \, (f_n | \ln f_n| + f_n ) \, dv  \\
&   \quad\le C \norm{ f_n * |\cdot|^{-2} 1_{|\cdot| \le 1}}_{L^4}  \norm{ \big[f_n (\ln f_n)^+ + f_n\langle v \rangle^2 + e^{-1} M(v) \langle v \rangle^2 + f_n \big] \langle v \rangle^2 }_{L^{4/3}} \\
&\quad \quad+ \,C \,  \norm{ f_n  }_{L^1}\, \int_{\R^3}  
\langle v \rangle^2 \,  (f_n | \ln f_n| + f_n ) \, dv  \\
& \quad\le  C \norm{ f_n  }_{L^2}\Big( \norm{[f_n^{4/3} [(\ln f_n)^+]^{4/3}\langle v \rangle^{8/3} + f_n^{4/3}\langle v \rangle^{16/3}}_{L^{1}}^{3/4} + 1 \Big) \\
&\quad \quad+ C   \norm{ f_n  }_{L^1}\norm{[f_n (\ln f_n)^+ + f_n\langle v \rangle^2 + e^{-1} \mathcal M(v) \langle v \rangle^2 + f_n ] \langle v \rangle^2}_{L^1} \\
&\quad \le  C \norm{ f_n  }_{L^2} \Big( \norm{ f_n}_{L^{2}_{4/3}}^{3/2} +  \norm{f_n}_{L^2_2}^{1/2}  \norm{f_n}_{L^1_6}^{1/2}  + 1 \Big) \\
&\quad\quad + C   \norm{ f_n  }_{L^1} \Big(\norm{f_n}_{L^2_{1}}^2 +  \norm{f_n}_{L^1_4} + 1\Big),
\end{aligned}
\end{equation}
where we used Young's inequality for convolutions, the elementary inequalities \eqref{ei} and $(\ln x)^+ \le C x^{1/2}$, together with the
interpolation inequality
\begin{equation}\label{eii}  \norm{g}_{L_4^{4/3}} \le  \norm{g}_{L^2_2}^{1/2}\,  \norm{g}_{L^1_6}^{1/2} ; 
\end{equation}
fourth, we obtain
\begin{equation}\label{e5}
\begin{aligned}
 \abs{ \int_{\R^3} f_n\, A_n[f_n] : \nabla \otimes \nabla (\langle v \rangle^3 |v|^2) \, dv } &\le C  \int_{\R^3} f_n\, |A_n[f_n]|  
\, \langle v \rangle^3\, dv \\
&\le   C\big (  \norm{f_n}_{L^1} +  \norm{f_n}_{L^2}\big)    \norm{ f_n  }_{L^1_3} ; 
\end{aligned}
\end{equation}
and fifth (using again interpolation \eqref{eii} in the last inequality)
\begin{equation}\label{e6}
\begin{aligned}
& \abs{ \int_{\R^3} f_n\, b_n[f_n] \cdot \nabla (\langle v \rangle^3 |v|^2) \, dv }  \\
 &\qquad\le C  \int_{\R^3} f_n\, |b_n[f_n]|\langle v \rangle^4\, dv \\
& \qquad\le  C \int_{\R^3}  
\langle v \rangle^4 \, (f_n * |\cdot|^{-2} 1_{|\cdot| \le 1} ) \, f_n  \, dv  +  C \int_{\R^3}  
\langle v \rangle^4 \, (f_n * |\cdot|^{-2} 1_{|\cdot| \ge 1} ) \, f_n  \, dv  \\ 
& \qquad  \le C\, \norm{ f_n * |\cdot|^{-2} 1_{|\cdot| \le 1} }_{L^4}  \,  \norm{ f_n \langle v \rangle^4 }_{L^{4/3}} +  C\,  \norm{ f_n  }_{L^1}\norm{f_n}_{L^1_4}  \\
  &\qquad\le C\, \norm{ f_n }_{L^2}  \,  \norm{ f_n}_{L^{1}_6}^{1/2}\,  \norm{f_n}_{L^2_2}^{1/2}  +  C\,  \norm{ f_n  }_{L^1}\, \norm{f_n}_{L^1_4} .
  \end{aligned}
\end{equation}
Finally, the last term is bounded by
\begin{equation}\label{e7}
\begin{aligned}
	 \frac{1}{n} \int_{\R^3}& \Delta f_n \left(\ln f_n + \abs{v}^2 \right) \langle v \rangle^3 \, dv \\
	 &\leq \frac{1}{n} \int_{\R^3}\left[  |f_n \ln f_n -f_n|\, |\Delta\langle v \rangle^3 | + f_n |\Delta \left(\langle v \rangle^3\abs{v}^2\right) | \right] \, dv\\
	 &\leq \frac{1}{n}\left( \norm{f_n}_{L^2_{1/2}}^2 +  \norm{f_n}_{L^1_3} \right).
  \end{aligned}
\end{equation}
Regrouping estimates \eqref{e1} to \eqref{e7}, using estimate (\ref{propamom1}) in Proposition \ref{existn} and integrating in $(s_1,s_2)$, we get estimate \eqref{estldeux1}. 
\end{proof}

Thanks to Theorem 1.1 in \cite{GS}, we know that the Fisher information for $f_n$  is bounded (more explicitly monotone decreasing in time). This bound, however, depends on the Fisher information of the initial data and will be  lost in the limit $n\to +\infty$. The next proposition shows that at  time $\min(t_1,t)$ (depending on $n$), the Fisher information is finite and not depending on $n$. Thanks to its monotonicity, this uniform bound will hold for all times larger than $\min(t_1,t)$, and therefore for all times larger than $t$. 



\begin{Prop} \label{mainprop}
Let $f_{in}:=f_{in}(v)\ge 0$ be as in Theorem \ref{main-Thm}. For any $t \in ]0, 1]$, the solution $f_n$ to the approximated Landau-Coulomb equation \eqref{eqa}, \eqref{in_data_approx}  built in Proposition \ref{existn} satisfies
\begin{equation}\label{estiprin}
 \sup_{ \sigma \in [t,+\infty)} \int_{\R^3}  \frac{ |\nabla f_n(\sigma,v)|^2 }{ f_n(\sigma,v)}\, dv \le C'\,  t^{-9/2},
 \end{equation}
where $C'$ depend only on  the mass, energy and entropy of $f_{in}$, and on $||f_{in}||_{L^1_{{{18}}}}$.
\end{Prop}
\begin{proof}
We first observe that (for any $g\ge 0$ such that the quantity exists)
\begin{equation}\label{eh1}
 H_3(g|\mathcal M) =  \int_{\R^3} \bigg[ \frac{g}{\mathcal M} \ln \bigg( \frac{g}{\mathcal M} \bigg) - \frac{g}{\mathcal M} + 1 \bigg] \mathcal M\langle v \rangle^3\, dv \ge 0, 
\end{equation}
and (using both inequalities of (\ref{ei}))
\begin{equation}\label{eh2}
\begin{aligned}
 H_3(g|\mathcal M) &\le  \int_{\R^3}  [\, g |\ln g| +  g\langle v \rangle^2 + \mathcal M(v) ] \langle v \rangle^3 \, dv \\
&\le C  \int_{\R^3}  [ g (\ln g)^+ +  g\langle v \rangle^2 + \mathcal M(v)\langle v \rangle^2  ] \langle v \rangle^3 \, dv  \\
&\le C \, (\norm{ g }_{L^2_{3/2}}^2  +  \norm{ g}_{L^{1}_5} + 1) .
\end{aligned}
\end{equation}
Let $t_0$ and $t_1$ be as in Propositions \ref{estldeux} and \ref{propaldeux} with $k= 9/4$.  We use \eqref{estldeux1} with $s_1 = t_0$ and $s_2 = \min(t_1,t)$. We get 
 \begin{equation*}
\begin{aligned}
 	H_3(f_n|M)(\min(t_1,t)) &+   c_0\, \int_{t_0}^{\min(t_1,t)} \int_{\R^3}  \frac{ |\nabla f_n(s,v)|^2 }{ f_n(s,v)}\, dv ds\\
		&\le  H_3(f_n|M)(t_0) +   K\, (\min(t_1,t) - t_0)\,  \sup_{s \in [t_0, \min(t_1,t)]}  \left(1 + \norm{f_n(s, \cdot)}_{L^2_2}\right)^3 . 
\end{aligned}
\end{equation*}
{{Thanks to the monotonicity of the Fisher information \cite[Theorem 1.1]{GS}, using estimates \eqref{eh1} and \eqref{eh2}, as well as \eqref{propamom1} from Proposition \ref{existn}, we see that 
 \begin{equation*}
\begin{aligned}
{ c_0} \int_{\R^3}&  \frac{ |\nabla f_n(\tau,v)|^2 }{ f_n(\tau,v)}\, dv ds\\
&  \le  \frac{C}{\min(t_1,t) - t_0} \,(1+ \norm{f_n(t_0, \cdot)}_{L^2_{3/2}}^2 ) +   K\,   \sup_{s \in [t_0, \min(t_1,t)]}  (1 + \norm{f_n(s, \cdot)}_{L^2_2})^3  \\
&  \le  \frac{C}{\min(t_1,t) - t_0} \,(1+ \norm{f_n(t_0, \cdot)}_{L^2_{9/4}}^2 ) +   K\,   \sup_{s \in [t_0, \min(t_1,t)]}  (1 + \norm{f_n(s, \cdot)}_{L^2_{9/4}})^3 ,
\end{aligned}
\end{equation*}
for $\tau = \min(t_1,t)$. 
}}
We now use Proposition \ref{estldeux} (with {{$k= 9/4$, so that $4 k + 9 = 18$}}), and get
 $$\norm{f_n(t_0, \cdot)}_{L^2_{9/4}} \le C_{9/4}^*\, t^{- 3/4}.$$
Then Proposition \ref{propaldeux} (applied with $k= 9/4$) yields 
\begin{align}\label{L^2_n}
 \sup_{s \in [t_0, \min(t_1,t)]} \norm{f_n(s, \cdot)}_{L^2_{9/4}} \le 2^{1/4}\, \norm{f_n(t_0, \cdot)}_{L^2_{9/4}} \le 2^{1/4}\, C_{9/4}^*\, t^{-3/4},
 \end{align}
and  we find 
 \begin{equation*}
\begin{aligned}
	\int_{\R^3}  \frac{ |\nabla f_n(\tau,v)|^2 }{ f_n(\tau,v)}\, dv &\le  \frac{1}{\min(t_1,t) - t_0}\, \int_{t_0}^{\min(t_1,t)} \int_{\R^3}  \frac{ |\nabla f_n(s,v)|^2 }{ f_n(s,v)}\, dv ds\\
& \le \frac{C}{\min(t_1,t) - t_0} \, \left(1+ (C_{9/4}^*\, t^{- 3/4})^2  \right) +   C  \left(1 + 2^{1/4}\, C_{9/4}^*\, t^{-3/4}\right)^3 .
\end{aligned}
\end{equation*}
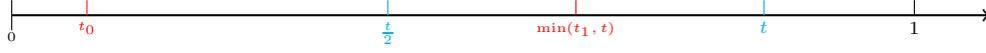
\begin{figure}
\begin{tikzpicture}[scale=1]
  \draw[thick,black, ->] (0, 0)  -- (13, 0) ;
  \draw (0,-0.2) -- (0, 0.2) node[below=9pt] {\textcolor{black}{\tiny{$ 0 $}}};
  \draw[red] (1,0) -- (1, 0.2)  node[below=5pt] {\textcolor{red}{\tiny{$ t_0 $}}};
    \draw[red] (7.5,0) -- (7.5, 0.2)  node[below=5pt] {\textcolor{red}{\tiny{$\min(t_1,t)$}}};
   \draw[cyan] (5,0) -- (5, 0.2) node[below=5pt] {\footnotesize{\textcolor{cyan}{$\frac{t}{2}$}}};
    \draw[cyan] (10,0) -- (10, 0.2) node[below=5pt] {\footnotesize{\textcolor{cyan}{$t$}}};
   \draw (12,0) -- (12, 0.2) node[below=5pt]{\footnotesize{$1$}};
\end{tikzpicture}
\caption{The timeline. Note that $t$ is independent of $n$, whereas $t_0$ and $t_1$ depend on $n$.}
\end{figure}
We note that (remember that $t \in ]0,1]$)
\begin{align*}
t_1 - t_0 
&= (2C_{9/4})^{-1} \, \ln \left( 1 + \frac{1}{1+ 2  \norm{f_n(t_0, \cdot) }_{L^2_{9/4}}^{4}}\right) \\
& \ge \left\{ (2C_{9/4})^{-1} \, \ln \left[ 1 + \frac{1}{1+ 2 \left(C_{9/4}^*\, t^{- 3/4}\right)^4}\right]\right\}  \ge D^{-1}\,{\ln \left(1 + \frac{t^3}{t^3 +C}\right)}\\
& \ge D^{-1}\,\ln \left(1 + \frac{t^3}{1 +C}\right)  \ge D^{-1}\,t^3,
\end{align*}
(where $C, D$ are constants regrouping previous constants),
and (since $t - t_0 \ge \frac{t}2$)
$$ \min(t_1,t) - t_0 \ge \min( D^{-1}\,t^3 , \frac{t}2) \ge  D^{-1}\,t^3, $$
so that
$$  \int_{\R^3}  \frac{ |\nabla f_n(\tau,v)|^2 }{ f_n(\tau,v)}\, dv 
\leq \frac{K}{t^3}  (1 + K\, t^{-3/2}) + K\, (1 + K \,t^{-3/4})^3 \le K\,t^{- 9/2}.$$
Thanks to the monotonicity of the Fisher information \cite[Theorem 1.1]{GS}, we then conclude that 
\begin{align}\label{mono_Fish_n}
\sup_{ \sigma \in [\min (t_1,t), +\infty)}  \int_{\R^3}  \frac{ |\nabla f_n(\sigma,v)|^2 }{ f_n(\sigma, v)}\, dv \le   K \,t^{-9/2},
\end{align}
which implies \eqref{estiprin}.

\end{proof}

\subsection{The limit}


We now finish the proof of Theorem \ref{main-Thm} by taking the limit $n\to +\infty$. We restate here Theorem \ref{main-Thm}  in a slightly more precise way, presenting the
\begin{Thm} \label{mainthm}
Let $f_{in}:=f_{in}(v)\ge 0$ lying in $L^1_{18}(\R^3)$ be such that   $H(f_{in})$ is finite. Then, there exists a (global in time) $H$-solution (also weak solution in the sense given by \cite{FPL}) $f$ to the Landau-Coulomb equation \eqref{Landau_orig}
 with initial datum $f_{in}$.
 Moreover this solution is strong, and belongs to $C^{\infty} \cap L^1_{18}(\R^3)$) on any time interval $(0, +\infty)$.
Finally, $f$ satisfies for all $t>0$  the bound: 
$$ \sup_{ \sigma \in [t, +\infty)}  \int_{\R^3}  \frac{ |\nabla f(\sigma,v)|^2 }{ f(\sigma, v)}\, dv \le   C' \,t^{-9/2} , $$
where $C'$ depends only on the mass and energy of $f_{in}$, $H(f_{in})$ and $\norm{f_{in}}_{L^1_{18}}$.
\end{Thm}

\begin{proof}
Up to subsequences, the sequence of solutions $f_n$ found in Proposition  \ref{existn} converges, as $n \to +\infty$ to a $H$-solution $f$ to the Landau-Coulomb equation (see \cite{V98} Section 6 for details), which satisfies the following weak formulation 
\begin{align*}
&\int f_{in} \varphi(\cdot,0) \;dv + \int_0^T \int f \partial_t \varphi \;dvds \\
&= \int_0^T \iint \left( \frac{f(v)f(w)}{|v-w|}\right)^{1/2} (\nabla_v \varphi - \nabla_w \varphi) \cdot \Pi(v-w)(\nabla_v - \nabla_w)\left( \frac{f(v)f(w)}{|v-w|}\right)^{1/2} \;dvdwds ,
\end{align*}
for all test functions $\varphi \in C^2_0([0,T)\times \mathbb{R}^3)$. 
This $H$-solution is also a weak solution in the sense of Corollary 11 in \cite{FPL}. 
Moreover $f_n \to f$ almost everywhere in time and velocity thanks to Aubin-Lions lemma, $f\in C_w([0,T], L^1(\R^3))$, and, 
thanks to Fatou lemma, using (\ref{estiprin}) (and the link between $ \frac{ |\nabla f_n|^2 }{ f_n}$ and $|\nabla \sqrt{f_n}|^2$),
that is (for $t>0$)
\begin{equation*}
 \sup_{[t,+\infty)} \int_{\R^3}  { |\nabla \sqrt{f_n}|^2 }\, dv \le C'\,t^{-9/2},
 \end{equation*}
we see that
 \begin{align*}
  \sup_{[t,+\infty)} \int_{\R^3}  { |\nabla \sqrt{f}|^2 }\, dv \le C'\, t^{-9/2}.
 \end{align*}
At this point, classical arguments imply that $f$ is bounded for all times greater than $t$ and therefore smooth \cite{S17, ABDL3, GG29} (and satisfies the equation in a strong sense).
\end{proof}

\appendix

\section{Coercivity for the truncated kernel}\label{appendix}

{\em{Proof of \eqref{coer1}}}:  Let $f$ be a function with bounded mass, second moment and entropy. Following the same computations as in Lemma 3.1 in \cite{S17}, we can show that there exists $\ell, \; R, \; \mu$ not depending on $n$ such that the measure of the set 
$$
B:= \{ v\in B_R \; | \; f \ge \ell \} 
$$
is smaller (or equal) than $\mu$. 
Let us now show \eqref{coer1}; since $K_n \le K_{n+1}$, it is enough to show the lower bound for $K_1$. Let $e$ be an unitary vector. We have 
$$
A_1[g](v) : e \otimes e = \int f(v-w)K_1(w) ( 1 - \cos ^2 \theta)dw,
$$
where $\theta$ is the angle between $w$ and $e$. Hence 
$$
A_1[g](v) : e \otimes e \ge   \ell \int_{ \{w-v \in B_R \} \cap \{f(v-w)\ge \ell\} }  K_1(w) ( 1 - \cos ^2 \theta)dw,
$$

Since $( 1 - \cos ^2 \theta) \sim 0$ when $w$ is parallel to $e$, and since $ w$ belongs to a subset of $B_R+v$, the smallest value for $A_1[g](v) : e \otimes e $ is obtained when $e = \frac{v}{|v|}$. Let $v$ be in a neighborhood of the origin. Then one can find a cone with center at the origin and aperture $\varepsilon$ in direction $e = \frac{v}{|v|}$ such that the intersection of this cone with the set $\{B+v\}$ has measure smaller than $\mu/4$. In this set $( 1 - \cos ^2 \theta) \ge \varepsilon^2$ and 
$$
A_1[g](v) : e \otimes e \ge  \ell \int_{\{w-v \in B_R  \} \cap \{f(v-w)\ge \ell\}\cap \textrm{cone}_\varepsilon} K_1(w) ( 1 - \cos ^2 \theta)dw \ge c(\ell,R,\mu).
$$
If $v$ is away from the origin, then $K(w) = \frac{1}{|w|}$ in $\{B+v \} $ and 
$$
A_1[g](v) : e \otimes e \ge  \ell \int_{\{w-v \in B_R \} \cap \{f(v-w)\ge \ell\}\cap \textrm{cone}_\varepsilon} \frac{1}{|w|} ( 1 - \cos ^2 \theta)dw \ge c(\ell,R,\mu) \ge \frac{c_0}{(1+ |v|^3)},
$$
as in the classical case. 




%


\bigskip

\noindent
\begin{thebibliography}{10}


\bibitem{AB}
A. A. Arsenev, O. E. Buryak. 
\newblock On the connection between a solution of the Boltzmann equation and a solution of the Landau-Fokker-Planck equation.
 \newblock{\em Math. USSR Sb. 69 (1991), no. 465. }
 
 

%
 
\bibitem{ABDL2}
 R. Alonso, V. Bagland, L. Desvillettes, and B. Lods: 
\newblock About the Landau-Fermi-Dirac equation with moderately soft potentials.
\newblock
{\em Arch. Rational Mech. Anal.,} {\bf{244}}, (2022), 779--875.

\bibitem{ABDL3}
 R. Alonso, V. Bagland, L. Desvillettes, and B. Lods: 
\newblock A priori estimates for solutions to Landau equation under Prodi-Serrin like criteria.
\newblock To appear  in  Arch. Rational Mech. Anal.

\bibitem{ABL}
 R. Alonso, V. Bagland, and B. Lods:
\newblock Uniform estimates on the Fisher information for solutions to Boltzmann and Landau equations.
\newblock {\em Kinet. Relat. Models}, {\bf 12} (2019), {1163--1183}.

\bibitem{CDH}
K. Carrapatoso, L. Desvillettes and L. He.
\newblock Estimates for the large time behavior of the Landau equation in the Coulomb case.
\newblock {\em Arch. Rational Mech. Anal.}, {\bf 224} (2017), {381--420}.
%

\bibitem{FPL} 
L.~Desvillettes.
\newblock Entropy dissipation estimates for the Landau equation in the Coulomb case
and applications.
\newblock {\em J. Funct. Anal.}, {\bf{269}} (2015), 1359--1403. 


\bibitem{GGIV} 
F. Golse, M. Gualdani, C. Imbert, A. Vasseur. 
\newblock Partial Regularity in time for the space homogeneous Landau equation with Coulomb potential. 
\newblock {\em Ann. Sci. Ec. Norm. Super. (4) 55 (2022), no. 6, 1575-1611.}

\bibitem{GG29}
M. Gualdani, N. Guillen.
\newblock On Ap weights and the Landau equations.
\newblock {\em Calc. Var. Partial Differential Equations 58 (2019), no. 1, 58:17.}



\bibitem{GS}
N. Guillen, L. Silvestre.
\newblock The Landau equation does not blow up.
\newblock Arxiv:2311.09420.

\bibitem{GGL}
W. Golding, M. Gualdani, A. Loher.
\newblock Global smooth solutions to the Landau-Coulomb equation in $L^{3/2}$.
\newblock Arxiv:2401.06939.

\bibitem{GL}
W. Golding,  A. Loher.
\newblock Local-in-time strong solutions of the homogeneous 
Landau-Coulomb equation with $L^p$ initial datum.
\newblock Arxiv:2308.10288.


\bibitem{Ji}
S. Ji.
\newblock Dissipation estimates of the Fisher information for the Landau equation.
\newblock Arxiv 2410.09035.

\bibitem{S17}
L. Silvestre. 
\newblock Upper bounds for parabolic equations and the Landau equation.
\newblock {\em J. Differential Equations 262 (2017), no. 3, 3034 - 3055.}

\bibitem{V98}
C. Villani. 
\newblock On a new class of weak solutions to the spatially homogeneous Boltzmann and Landau equations. 
\newblock {\em Arch. Ration. Mech. Anal. 143 (1998), 273 - 307.}


\bibitem{V96}
C. Villani. 
\newblock On the Landau equation: weak stability, global existence. 
\newblock {\em Adv. Diﬀ. Eq.1 (1996), no. 5, 793-816.}





\end{thebibliography}
\end{document}